\documentclass[10pt]{amsart}

\usepackage{amsmath,amsfonts, latexsym,graphicx, footmisc, amssymb, mathtools, enumerate, mdwlist}

\usepackage{hyperref}
\hypersetup{colorlinks=true, urlcolor=blue, citecolor=blue, linkcolor=blue}

\newcommand{\U}{{\mathcal U}}
\newcommand{\0}{{\mathbf 0}}
\newcommand{\C}{{\mathbb C}}
\newcommand{\Z}{{\mathbb Z}}

\newcommand{\cL}{{\mathbb L}}

\newcommand{\rank}{\mathop{\rm rank}\nolimits}

\newcommand{\Pdot}{\mathbf P^\bullet}

\newcommand{\vdual}{{\mathcal D}}

\newcommand{\mult}{{\operatorname{mult}}}

\newtheorem{defn0}{Definition}[section]
\newtheorem{prop0}[defn0]{Proposition}
\newtheorem{conj0}[defn0]{Conjecture}
\newtheorem{thm0}[defn0]{Theorem}
\newtheorem{lem0}[defn0]{Lemma}
\newtheorem{corollary0}[defn0]{Corollary}
\newtheorem{example0}[defn0]{Example}
\newtheorem{remark0}[defn0]{Remark}
\newtheorem{question0}[defn0]{Question}
\newtheorem{exercise0}[defn0]{Exercise}

\newenvironment{thm}{\begin{thm0}}{\end{thm0}}
\newenvironment{lem}{\begin{lem0}}{\end{lem0}}
\newenvironment{cor}{\begin{corollary0}}{\end{corollary0}}

\newenvironment{exm}{\begin{example0}\rm}{\end{example0}}
\newenvironment{rem}{\begin{remark0}\rm}{\end{remark0}}

\newcommand{\thmref}[1]{Theorem~\ref{#1}}
\newcommand{\lemref}[1]{Lemma~\ref{#1}}
\newcommand{\corref}[1]{Corollary~\ref{#1}}

\newcommand{\secref}[1]{Section~\ref{#1}}

\newcommand{\mbf}[1]{{\mathbf #1}}

\title[Relative Polar Multiplicities and the Real Link]{Relative Polar Multiplicities and the Real Link}


\author{David B. Massey}

\date{}

\begin{document}

\begin{abstract} For a hypersurface defined by a complex analytic function, we obtain a chain complex of free abelian groups, with ranks given in terms of relative polar multiplicities, which has cohomology isomorphic to the reduced cohomology of the real link. This leads to Morse-type inequalities between the Betti numbers of the real link of the hypersurface and the relative polar multiplicities of the function.
\end{abstract}

\maketitle




\section{Introduction}

Let $\U$ be an open neighborhood of the origin in $\C^{n+1}$, where $n\geq 1$, and let $f: (\U,\0)\rightarrow (\C,0)$ be a complex analytic function which is not locally constant at $\0$. We let $X:=V(f)$ and let $\Sigma f$ denote the critical locus of $f$. To eliminate the trivial case, we assume that $\0\in\Sigma f$ and let $s:=\dim_\0\Sigma f$. We shall always use $\Z$ for our homology/cohomology coefficients.

\smallskip

In the now-classic 1968 book of Milnor  \cite{milnorsing}, the  {\it real link} of $X$ at $\0$, $K_X$, plays a prominent role (in fact, the real link, though not by name, appears in \L ojasiewicz's 1965 book \cite{lojbooknotes} in a more-general context). Milnor proves in Theorem 5 of \cite{milnorsing} that $K_X$ is $(n-2)$-connected; hence, the  first interesting cohomology group of the real link is $H^{n-1}(K_X)$.

The {complex link} of a complex analytic space at a point was studied in depth by L\^e and Kato in 1975 in \cite{lk}, and for complete intersections by L\^e in 1979 in \cite{levan}. The complex link, and its topological relationship to the real link, was discussed at length in Goresky and MacPherson's 1988 book ``Stratified Morse Theory'' \cite{stratmorse}.

For our hypersurface $X$,  we let $\cL_X$ denote the complex link of $X$ at $\0$. In \cite{levan}, L\^e proves that  $\cL_X$  has the homotopy-type of a finite bouquet of $(n-1)$-spheres and that the number of $(n-1)$-spheres in the bouquet is given by the intersection number $\gamma^1_{f}:=\big(\Gamma^1_{f,L}\cdot V(L)\big)_\0$, where $\Gamma^1_{f,L}$ is the relative polar curve of $f$ with respect to a generic linear form $L$ (see \cite{hammlezariski}, \cite{teissiercargese}, and \cite{lecycles}).

On the cohomological level, it is known, but perhaps not so well-known that  there is an injection of reduced cohomology groups 
$$
\widetilde H^{n-1}(K_X)\hookrightarrow\widetilde H^{n-1}(\cL_X)\cong  \Z^{\gamma^1_{f}}.\leqno{(\dagger)}
$$
For $s=0$, see Corollary 3.3 of \cite{takeuchiperv} or  Proposition 6.1.22 of \cite{dimcasheaves};  for general $s$, see Corollary 4.10 of \cite{milnorfiberlink}. 

\bigskip

There are higher-dimensional relative polar multiplicities $\gamma^1_f, \dots, \gamma^n_f$; we shall define these in \secref{sec:prior}, but note now that $\gamma^0_f=0$, $\gamma^{n+1}_f=1$, and $\gamma^n_f=-1+\mult_\0f$.

By combining and/or modifying several of our previous results, and using these higher-dimensional multiplicities, we quickly obtain in this short paper a surprising generalization of ($\dagger$). This generalization is the main theorem in this paper:

\medskip
\vbox{
\noindent{\bf Theorem} (\thmref{thm:chaincomplex2}):

\smallskip

{\it  Let $\lambda^0_X:=\gamma^0_f+\gamma^1_f=\gamma^1_f$, \ $\lambda^1_X:=\gamma^1_f+\gamma^2_f$, \ $\lambda^2_X:=\gamma^2_f+\gamma^3_f$, $\dots$, $\lambda^{n-1}_X:=\gamma^{n-1}_f+\gamma^n_f$, \ and $\lambda^n_X:=\gamma^n_f+\gamma^{n+1}_f=\gamma^n_f+1=\mult_\0 f$. 

\smallskip

Then, there is a chain complex
$$
0\rightarrow \Z^{\lambda^0_X}\rightarrow \Z^{\lambda^{1}_X}\rightarrow \Z^{\lambda^{2}_X}\rightarrow\cdots\rightarrow  \Z^{\lambda^{n-1}_X}\rightarrow  \Z^{\lambda^n_X}\rightarrow 0
$$
such that the cohomology at the $\lambda^k_X$ term is isomorphic to $\widetilde H^{n+k-1}(K_X)$.
}
}

\medskip

Of course, this theorem implies that $\widetilde H^{n-1}(K_X)$ injects into $\Z^{\gamma^1_{f}}$. It also implies  immediately that one has the following Morse-type inequalities.

\medskip

\noindent{\bf Corollary} (\corref{cor:morselinkinequal}):

\smallskip

{\it For all $k$, let $\tilde b^k:=\rank \widetilde H^k(K_X)$. Then, for all $p$ such that $0\leq p\leq n$, we have the following inequalities:

\medskip

\begin{enumerate}
\item 
$$
(-1)^p\sum_{k=0}^p(-1)^k\tilde b^{n+k-1}\leq (-1)^p\sum_{k=0}^p(-1)^k\lambda^k_X=\gamma^{p+1}_f \textnormal{ and }
$$
\medskip
\item 
$$
(-1)^p\sum_{k=0}^p(-1)^k\tilde b^{2n-k-1}\leq (-1)^p\sum_{k=0}^p(-1)^k\lambda^{n-k}_X=(-1)^p+\gamma^{n-p}_f. 
$$
\end{enumerate}
}

\medskip

\section{Prior definitions and results}\label{sec:prior}

We continue with our notation from the introduction:  $X:=V(f)$ is the hypersurface defined by $f$, $s=\dim_\0\Sigma f$, $K_X$ is the real link of $X$ at $\0$, and $\cL_X$ is the complex link of $X$ at $\mbf 0$.

\medskip

For all $k$ such that $0\leq k\leq n+1$, we  need relative polar varieties/cycles $\Gamma^k_f$, where $k$ is the dimension, as developed by L\^e, Hamm, and Teissier. We have written about these in many, many places; see, for instance, \cite{lecycles} and Example 6.10 of \cite{numinvar}. Recall that, for a linear choice of coordinates $\mbf z:=(z_0, \dots, z_n)$ for $\C^{n+1}$, as a set, 
$$\Gamma^k_{f,\mbf z} \ = \ \overline{V\left(\frac{\partial f}{\partial z_k}, \dots, \frac{\partial f}{\partial z_n}\right)-\Sigma f};
$$ 
this will be purely $k$-dimensional provided that $\mbf z$ is generic enough. We give this a cycle structure by multiplying each irreducible component $C$ by the Milnor number 
$$
\mu_{\mbf p}(f_{|_{V(z_0-p_0, \dots z_{k-1}-p_{k-1})}}),
$$
for a generic point $\mbf p=(p_0, \dots, p_n)\in C$; this yields the {\it $k$-dimensional relative polar cycle of $f$, with respect to $\mbf z$}, which we write as $\Gamma^k_{f,\mbf z}$.  If $\Gamma^k_{f,\mbf z}$ is $k$-dimensional, we define $\gamma^k_{f, \mbf z}$ (or $\gamma^k_{f, \mbf z}(\0)$) to be the intersection number $\big(\Gamma^k_{f,\mbf z}\cdot V(z_0, \dots z_{k-1})\big)_\0$ provided that the intersection is proper.

Corollaire IV.5.4.3 of \cite{teissiervp2} tells us that, for generic $\mbf z$, $\gamma^k_{f, \mbf z}=\mult_\0 \Gamma^k_{f,\mbf z}$ and is independent of the generic choice of $\mbf z$. Thus, we write simply $\gamma^k_f$ for this common generic value; this is the {$k$-dimensional relative polar multiplicity of $f$}. As we mentioned in the introduction, $\gamma^0_f=0$, $\gamma^{n+1}_f=1$, and $\gamma^n_f=-1+\mult_\0f$.

\medskip

Using the Wang sequence of Lemma 8.4 of \cite{milnorsing}, together with the result of Kato and Matsumoto in \cite{katomatsu} that the Milnor fiber in $(n-s-1)$-connected and Alexander duality, we obtain the well-known fact that $H^k(K_X)=0$ unless $k=0$, $k=2n-1$, or $n-1\leq k\leq n+s$. In addition, it is also well-known that the rank of $H^{2n-1}(K_X)$ is equal to the number of irreducible components of $X$ at the origin (as $K$ is compact, this follows at once from Lemma 19.1.1 of \cite{fulton}).

\medskip

\medskip

Now, we need to recall notation/results from \cite{numinvar} in which we used the derived category (of bounded, constructible complexes of sheaves of $\Z$-modules). For generic linear coordinates $\mbf z$, for $0\leq k\leq n$, we use iterated nearby and vanishing cycles to define
$$
\lambda^k_X:=\rank H^0\big(\phi_{z_k}[-1]\psi_{z_{k-1}}[-1]\dots \psi_{z_{0}}[-1]\Z_X^\bullet[n]\big)_\0.
$$
Note that we do not distinguish in our notation between $z_i$ and the restriction  of $z_i$ to subspaces. In the notation of \cite{numinvar}, $\lambda^k_X$ would be $\lambda^k_{\Z^\bullet_X}$. 
\bigskip

We combine a number of our previous results to obtain: 

\medskip

\begin{thm}\label{thm:chaincomplex1} We have the following equalities $\lambda^0_X=\gamma^0_f+\gamma^1_f=\gamma^1_f$, \ $\lambda^1_X=\gamma^1_f+\gamma^2_f$, \ $\lambda^2_X=\gamma^2_f+\gamma^3_f$, $\dots$, $\lambda^{n-1}_X=\gamma^{n-1}_f+\gamma^n_f$, \ and $\lambda^n_X=\gamma^n_f+\gamma^{n+1}_f=\gamma^n_f+1=\mult_\0 f$, and there is a chain complex
$$
0\rightarrow \Z^{\lambda^n_X}\rightarrow \Z^{\lambda^{n-1}_X}\rightarrow \Z^{\lambda^{n-2}_X}\rightarrow\cdots\rightarrow  \Z^{\lambda^1_X}\rightarrow  \Z^{\lambda^0_X}\rightarrow 0
$$

\medskip

\noindent such that the homology/cohomology at the $\lambda^k_X$ term is isomorphic to the stalk cohomology at the origin in degree $-k$ of $\Z^\bullet_X[n]$. Thus, the homology/cohomology at the $\lambda^k_X$ term is isomorphic to $\Z$ when $k=n$ and is zero when $k\neq n$.
\end {thm}
\begin{proof} The existence of the chain complex with the given stalk cohomology appears in Theorems 5.3 and 5.4 of \cite{numinvar}, but with different shifts; for us now, we would use $\Pdot:=\Z^\bullet_X[n]$ and include shifts by $-1$ on the nearby and vanishing cycles. 

That the modules are actually free abelian follows from L\^e's result on the homotopy-type of the complex links of strata for a local complete intersection in \cite{levan}, together from the results in \cite{singenrich}, but is more clear from Definition 2.11, Corollary 6.4, and Corollary 8.4 in \cite{calcchar}. 

Finally, the relation between the $\lambda^k_X$ and the $\gamma^k_f$ is given in Example 8.4 of \cite{numinvar}.
\end{proof}

\medskip

\begin{rem}
Whenever one has a chain complex where the terms in the complex have finite rank, there are always Morse inequalities between the ranks of the terms in the complex and the ranks of the cohomology of the complex. In a more-general context, this is presented in Corollary 5.5 of \cite{numinvar}. However, the Morse inequalities which follow from \thmref{thm:chaincomplex1} have no real content; the reader should check that the Morse inequalities yield only that, for all $k$, $\gamma^k_f\geq 0$.
\end{rem}

\bigskip

\section{Morse inequalities and the real link}

The reader should be wondering why we bothered discussing Morse inequalities related to \thmref{thm:chaincomplex1} when we concluded that the related Morse inequalities told us essentially nothing. The reason appears below, where we ``dualize'' the chain complex from \thmref{thm:chaincomplex1} and look at the associated Morse inequalities to yield non-trivial results about the Betti numbers of the real link of $X$.

\medskip

We continue to use $\lambda^0_X=\gamma^1_f$, \ $\lambda^1_X=\gamma^1_f+\gamma^2_f$, \ $\lambda^2_X=\gamma^2_f+\gamma^3_f$, $\dots$, $\lambda^{n-1}_X=\gamma^{n-1}_f+\gamma^n_f$, \ and $\lambda^n_X=\gamma^n_f+1=\mult_\0 f$. 

We let $m_\0:\{\0\}\hookrightarrow X$ be the inclusion, so that
$$H^k(m_\0^!\Z^\bullet_X[n])\cong H^{k+n}(B_\epsilon^\circ\cap X, B_\epsilon^\circ\cap X\backslash\{\0\})\cong \widetilde H^{n+k-1}(K_X),$$
where $B_\epsilon^\circ$ is a small open ball, centered at the origin.

\medskip

\begin{thm}\label{thm:chaincomplex2}   There is a chain complex
$$
0\rightarrow \Z^{\lambda^0_X}\rightarrow \Z^{\lambda^{1}_X}\rightarrow \Z^{\lambda^{2}_X}\rightarrow\cdots\rightarrow  \Z^{\lambda^{n-1}_X}\rightarrow  \Z^{\lambda^n_X}\rightarrow 0
$$
such that the cohomology at the $\lambda^k_X$ term is isomorphic to the costalk cohomology at the origin in degree $k$ of $\Z^\bullet_X[n]$. 

To be precise, the cohomology at the $\lambda^k_X$ term is isomorphic to $H^k(m_\0^!\Z^\bullet_X[n])\cong \widetilde H^{n+k-1}(K_X)$.
\end {thm}
\begin{proof} As was shown just before Theorem 5.4 of \cite{numinvar}, the complex in \thmref{thm:chaincomplex1} is obtained by starting with $\Z^\bullet_X[n]$ and then applying iterated nearby and vanishing cycles, repeatedly using the canonical map from $\psi_{z_k}[-1]$ to  $\phi_{z_k}[-1]$. One obtains the ``dual'' complex in the current theorem by instead repeatedly using the variation map from $\phi_{z_k}[-1]$ to  $\psi_{z_k}[-1]$. 

Alternatively one could begin with the Verdier dual $\vdual(\Z^\bullet_X[n]\big)$ obtain the chain complex analogous to that in \thmref{thm:chaincomplex1} by repeatedly using the canonical map from $\psi_{z_k}[-1]$ to  $\phi_{z_k}[-1]$, then Verdier dualizing, and finally using that Verdier dualizing commutes with $\psi_{z_k}[-1]$ and  $\phi_{z_k}[-1]$ (see \cite{natcommute}).
\end{proof}

\medskip

\begin{rem} Suppose that $n=1$; this special case is easy to analyze. \thmref{thm:chaincomplex2} implies that we have an exact sequence
$$
0\rightarrow \widetilde H^0(K_X)\rightarrow \Z^{-1+\mult_\0 f}\rightarrow  \Z^{\mult_\0 f}\rightarrow \widetilde H^1(K_X)\rightarrow 0.
$$
From this we conclude two obvious things: the reduced Euler characteristic of $K_X$ is $-1$, i.e., the Euler characteristic is $0$; and the number of connected components of $K_X$ is at most $\mult_0 f$.
\end{rem}

\medskip

The following corollary is immediate from the theorem.

\begin{cor}\label{cor:inject} $\widetilde H^{n-1}(K_X)$ injects into $\widetilde H^{n-1}(\cL_X)\cong\Z^{\gamma^1_f}$.
\end{cor}

\medskip

Now we want to look at the Morse inequalities derived from \thmref{thm:chaincomplex2} and obtain highly non-trivial inequalities.

\medskip

First, we need a trivial lemma, whose proof by ``telescoping'' we leave to the reader:

\begin{lem}\label{lem:telescope} For all $p$ such that $0\leq p\leq n$,

\medskip

\begin{enumerate}
\item $\sum_{k=0}^p(-1)^k\lambda^k_X=(-1)^p\gamma^{p+1}_f$, and
\medskip
\item $\sum_{k=0}^p(-1)^k\lambda^{n-k}_X=1+(-1)^p\gamma^{n-p}_f$.
\end{enumerate}

\medskip

\noindent(Recall that $\gamma^0_f=0$ and $\gamma^{n+1}_f=1$.)
\end{lem}

\medskip

Now, from the theorem and the above lemma, we immediately conclude:

\smallskip

\begin{cor}\label{cor:morselinkinequal}\textnormal{(Morse Link Inequalities)} For all $k$, let $\tilde b^k:=\rank \widetilde H^k(K_X)$. Then, for all $p$ such that $0\leq p\leq n$, we have the following inequalities:

\medskip

\begin{enumerate}
\item 
$$
(-1)^p\sum_{k=0}^p(-1)^k\tilde b^{n+k-1} \ \leq  \ (-1)^p\sum_{k=0}^p(-1)^k\lambda^k_X \ = \ \gamma^{p+1}_f \textnormal{ and }
$$
\medskip
\item 
$$
(-1)^p\sum_{k=0}^p(-1)^k\tilde b^{2n-k-1} \ \leq \ (-1)^p\sum_{k=0}^p(-1)^k\lambda^{n-k}_X \ = \ (-1)^p+\gamma^{n-p}_f. 
$$
\end{enumerate}
\end{cor}

\bigskip

\noindent One can also conclude from \thmref{thm:chaincomplex2} and \lemref{lem:telescope} that the Euler characteristic of $K_X$ is 0, but this is well-known.

\begin{exm}

Let $c$ denote the number of irreducible components of $X$ at $\0$; we remind the reader that this equals the rank of $H^{2n-1}(K_X)$. Note that, if $c\neq 1$, then $s=\dim_\0\Sigma f=n-1$.

\medskip

Now suppose that we take $p=0$ in \corref{cor:morselinkinequal}. Then we obtain two inequalities:
$$
\tilde b^{n-1}\leq \gamma^1_f\hskip 0.2in\textnormal{ and }\hskip 0.2in \tilde b^{2n-1}\leq 1+\gamma^n_f.
$$
This first inequality, which we have already discussed, is not obvious. However, the second inequality simply says that $c\leq\mult_\0 f$, which is obviously true.

\medskip

What if $p=1$? Then we obtain two inequalities:
$$
-(\tilde b^{n-1}-\tilde b^n)\leq \gamma^2_f\hskip 0.2in\textnormal{ and }\hskip 0.2in -(\tilde b^{2n-1}-\tilde b^{2n-2})\leq -1+\gamma^{n-1}_f.
$$
Again the first inequality is not obvious and, as far as we know, is a new bound. Is the second inequality obvious? That depends on the size of $s$. We remind the reader $H^k(K_X)=0$ unless $k=0$, $k=2n-1$, or $n-1\leq k\leq n+s$.  So, if $n+s+1\leq 2n-2$, i.e., if $s\leq n-3$ (and hence $c=1$), then $\tilde b^{2n-2}=0$ and the second inequality above becomes $-\tilde b^{2n-1}\leq -1+\gamma^{n-1}_f$, that is, $1=c=\tilde b^{2n-1}\geq 1-\gamma^{n-1}_f$, which is trivially true. On the other hand, if $\tilde b^{2n-2}\neq 0$, which requires that $s\geq n-2$, then the second inequality above is new as far as we know.
 
\end{exm}

\end{document}